\documentclass[english]{amsart}

\usepackage{babel}
\usepackage{amstext}
\usepackage{amsmath}
\usepackage{amscd}
\usepackage{amsfonts}
\usepackage{latexsym}
\usepackage{ifthen}
\newcommand{\CC}{\ensuremath{\mathbb{C}}}
\newcommand{\RR}{\ensuremath{\mathbb{R}}}
\newcommand{\ZZ}{\ensuremath{\mathbb{Z}}}

\newcommand{\hol}{\ensuremath{\mathcal{O}}}

\newcommand{\ra}{\ensuremath{\rightarrow}}

\newcommand{\lra}{\longrightarrow}

\newcommand\sF{{\mathcal F}}

\newcommand\sO{{\mathcal O}}

\newcommand\sC{{\mathcal C}}

\newcommand\bR{{\mathbb R}}
\newcommand\bZ{{\mathbb Z}}
\newcommand\bC{{\mathbb C}}

\newcommand\bP{{\mathbb P}}
\newcommand\rh{{\dasharrow}}

\newcounter{lemma}

\usepackage{enumerate}

\theoremstyle{plain} 
\newtheorem{theorem}{\noindent\bf Theorem}[section]
\newtheorem{lemma}[theorem]{\noindent\bf Lemma}
\newtheorem{corollary}[theorem]{\noindent\bf Corollary}
\newtheorem{proposition}[theorem]{\noindent\bf Proposition}

\theoremstyle{definition}
\newtheorem{remark}[theorem]{\noindent\bf Remark}

\title[Real tori] {On volume preserving complex structures  on real
tori}
\author{Fabrizio Catanese, Keiji Oguiso and Thomas Peternell}
\thanks{The present work took place in the realm of the DFG 
Forschergruppe 790 "Classification of algebraic
surfaces and compact complex manifolds". AMS Classification:  32Q55, 32J17, 32J18
}

\date{\today}

\begin{document}

\begin{abstract} A basic problem in the
classification theory of compact complex manifolds is to give simple characterizations of complex tori.
It is well known that a compact K\"ahler manifold  $X$ homotopically equivalent
to a a complex torus is biholomorphic to a complex torus.

The question whether a compact complex manifold  $X$ diffeomorphic
to  a complex torus is biholomorphic to a complex torus
has  a negative answer due to a construction by Blanchard and Sommese.

Their examples have however negative Kodaira dimension, thus it makes sense to ask the question
whether  a compact complex manifold  $X$ with trivial canonical bundle
which is homotopically equivalent
to  a complex torus is biholomorphic to a complex torus.

In this paper we show that the answer is positive for complex threefolds satisfying
 some additional condition, such as the existence of 
 a non constant meromorphic function. 
\end{abstract}

\maketitle
\tableofcontents
\section{Introduction}
\noindent

The Enriques-Kodaira  classification of
compact complex surfaces implies in particular that a compact complex surface homotopically 
equivalent to a complex
torus of dimension
$2$ is biholomorphic to a complex torus of dimension 2. 
The corresponding result in dimension 1 was 
 already known in the 19th century.

Surprisingly, the analogous result in dimension 3 is no longer true, as 
shown by  Sommese using  results of Blanchard 
(\cite{So75}, p.213, (E) after
\cite{Bl53}; see also
\cite{Ca02}, Section 5, \cite{Ca04}, Section 7). 

Indeed there are  countably many families of
complex manifolds even diffeomorphic to a complex torus of
dimension $3$, which are not biholomorphic to a  complex torus.

These are constructed as follows: let $L$ be a line bundle on a curve $C$,
generated by global sections (if $C$ is an elliptic curve, it suffices that the degree
of $L$ be at least $2$). Let $s_1, s_2 \in H^0(C,L)$ be two sections without common zeros,
so that $s : = (s_1,s_2)$ is a nowhere vanishing section of the rank two vector bundle
$ L \oplus L$. Identifying the fibre $\CC^2$ with the quaternions,
one finds that $s, is, js, ks$ yield four sections $ \in H^0(C,L\oplus L) $
giving an $\RR$ -basis over each point (hence the total space of $ L \oplus L$ is
diffeomorphic to a product $ C \times \RR^4$).

Defining $X$ as the quotient of the total space of $ L \oplus L$ by the free abelian subgroup
$\ZZ^4$ generated by the four sections, $X$ is then diffeomorphic to a torus,
yet its canonical bundle $K_X$ has the property  that $ K_X = p^* ( -2L)$ (\cite{Ca04}, Remark 7.3),
in particular $ h^0(X, - K_X) = h^0(C, 2L)$ which, in the case where $C$ is an elliptic curve,
equals $ 2 \ deg (L) \geq  4$. Hence $X$ is not a complex torus, for which $K_X$ is a trivial divisor.

It is now natural to ask which  kind of  additional conditions are
sufficient to characterize
complex tori  as complex
manifolds. The simplest among such conditions, under a weak
K\"ahler assumption (Theorem
\ref{thm0}) requires to have the same integral cohomology algebra.
However, if one drops the
K\"ahler condition, the problem becomes much more difficult and,
so far, not much is known
(see however some characterizations in
\cite{Ca95} and in \cite{Ca04}, especially Proposition 2.9).

The examples of Blanchard-Sommese lead one of the authors 
(\cite{Ca04} (p.269)) to ask the following
question 
\vskip .2cm \noindent
{\it Are there compact complex manifolds $X$ with
trivial canonical bundle $K_X$  which are diffeomorphic but not
biholomorphic to a complex torus?}

\vskip .2cm \noindent 
In response to this question we prove the following theorem as a corollary of more general results
(Theorems
\ref{thm0},
\ref{thm1},
\ref{thm2}, \ref{thm3}, \ref{thm4}):

\begin{theorem}\label{main}  Let $X$ be a compact complex manifold subject to the following conditions.
\begin{enumerate}
\item $X$ is homotopy equivalent to a
complex torus of dimension $3$;
\item $X$ has a dominant meromorphic map to a
compact complex  analytic space $Y$ of smaller dimension, i.e.,
with $0 < \dim\, Y < 3$;
\item $K_X \equiv 0 $, i.e., $\hol_X (K_X) \cong  \sO_X$.
\end{enumerate}
Then $X$ is biholomorphic to a complex torus. 
\end{theorem}

We should remark that  condition  (2) is of course not a necessary
condition for $X$ to be a complex torus.  However, the
only known examples of threefolds which are homeomorphic
 but not biholomorphic to a complex torus are Blanchard-Sommese's
examples and they
are all fibred  over elliptic
curves, as one can see from the construction described above.  So they satisfy
conditions (1) and (2) (but not (3)). 

\vskip .2cm \noindent
As a special case of Theorem \ref{main}, we obtain

\begin{corollary}\label{cor-torus}  Let $X$ be a compact complex
manifold such that:\\ (i) $X$ is homotopy equivalent
to a complex torus of dimension $3$;\\ (ii) $X$ has a non-trivial map
$\alpha : X \to T$  to a positive dimensional
complex torus $T$;\\ (iii) $K_X \equiv 0$.

Then $X$ is biholomorphic to a complex torus (of dimension $3$).
\end{corollary}

This is a direct consequence of Theorem \ref{main} using  (\cite{Ue75}, Lemma
10.1  and Theorem 10.3) in case $\dim \alpha(X) = 3.$ 

\begin{corollary}\label{cor-algdim}  Let $X$ be a compact complex
manifold such that
\begin{enumerate}
\item $X$ is homotopy equivalent
to a complex torus of dimension $3$;
\item either $a(X) > 0$, i.e.,
$X$ has a non-constant meromorphic function,  or the Albanese torus ${\rm Alb}\, (X)$ is non-trivial; 
\item  $K_X \equiv 0$.
\end{enumerate}
Then $X$ is biholomorphic to a complex torus (of dimension $3$).
\end{corollary}

The remaining case where $X$ has no non-constant meromorphic function and also no meromorphic
map  to a surface without meromorphic functions seems difficult. 

If however the tangent or the
cotangent bundle have some sections, the situation gets amenable:

\begin{theorem} Let $X$ be a smooth compact complex threefold with $K_X \equiv 0$ homotopy
equivalent to a torus.  If $h^0(T_X) \geq 3$ or
if $h^0(\Omega^1_X) \geq 3,$ then $X$ is biholomorphic to a torus. 
\end{theorem}

\par
\vskip 4pt
\noindent {\it Acknowledgements.} The second author is grateful to
the  guest  program of the Alexander von Humboldt
Foundation and to the DFG Forschergruppe 790 ``Classification of
algebraic surfaces and compact complex manifolds'', which
made this collaboration possible.

\section{Preliminaries}
\noindent

We start with some notations. Let $X$ be an irreducible compact complex space. Then
$a(X)$ denotes the algebraic dimension of $X$
(\cite{Ue75}, Definition 3.2), the maximal number of algebraically independent
 meromorphic functions. If $a(X) = \dim X$, 
then $X$ is called a Moishezon manifold.\\
We
also recall that for a compact complex manifold $X$, the Albanese
torus  of $X$ is the complex torus defined by
$${\rm Alb}\, (X) = H^0(X, d\sO_X)^{\vee}/\Lambda\,\, ,$$ where $\Lambda$
is the minimal closed complex Lie group
containing
${\rm Im}\, (H_1(X, \bZ) \to H^0(X, d\sO_X)^{\vee})$. 

We have then the
Albanese  morphism ${\rm alb}_X : X \to {\rm Alb}\,
(X)$ (see \cite{Ue75}, p.101--104), assigning to each point $x$ the class of the linear
functional $\int_{x_0}^x$ on $H^0(X, d\sO_X)$, obtained integrating on a path from $x_0$ to $x$.

\begin{proposition}\label{prop1} Let $f : X \longrightarrow Y$ be a
surjective morphism with connected fibres from a
compact (connected) complex manifold $X$ with
$\pi_1(X) \simeq \bZ^k$ to  a complex manifold $Y$. Let
$F$ be a general fibre of $f$. Then there exists an exact sequence of groups
$$0 \lra A \lra \pi_1(X) \simeq \bZ^{k} \buildrel {f_*} \over
{\lra} \pi_1(Y) \lra 0\,\, ,$$ where $A$ contains ${\rm Im}\, (\pi_1(F)
\to \pi_1(X))$ as a finite index subgroup.  In
particular, $\pi_1(Y)$ is a finitely generated abelian group of rank
$\le k$ and there is an inequality of Betti numbers 
$$b_1(F) + b_1(Y) \ge k = b_1(X)\, .$$                                        
\end{proposition}
                                                                         
\begin{proof} The following proof is very close to \cite{No83}, Lemma
1.5  and \cite{CKO03}, Lemma 3.

Let $F$ be a general fibre of $f$ and $U \subset Y$  be the maximal
Zariski open subset such that $f$ is smooth over
$U$.  Consider the following commutative diagram of  exact sequences:

\[
\begin{CD} 1 @>>> G @>>> \pi_1(f^{-1}(U)) @>{(f_U)_*}>> \pi_1(U)
@>>> 1\\ @VVV @VV{b_F}V @VV{b}V @VV{c}V @VVV\\ 1
@>>> {\rm Ker}\, f_* @>>> \pi_1(X) \simeq \bZ^{k}  @>{f_*}>> \pi_1(Y) @>>> 1
\end{CD}
\]
\par
\vskip 4pt
\noindent Here $G = {\rm Im}\,(\pi_1(F) \to \pi_1(f^{-1}(U)))$.  Since
$b$ is surjective, the snake lemma
yields the following exact sequence:
$${\rm Ker}\, b \to {\rm Ker}\, c \to  {\rm Coker}\, b_F \to 0\, . $$ Thus
$${\rm Ker}\, c/(f_{U})_*({\rm Ker}\, b) \simeq {\rm Coker}\, b_F  =
{\rm Ker}\, f_*/{\rm Im}\, b_F\, .$$  
Since ${\rm
Ker}\, f_* \subset \bZ^{k}$, it follows that 
$${\rm Ker}\, c/(f_{U})_*({\rm Ker}\, b) \simeq {\rm Coker}\, b_F$$ 
is a 
finitely generated abelian group. On the other  hand,  each $[\gamma]
\in {\rm Ker}\, c$ is represented by the product of conjugates of 
elements represented by a closed
circle
$\gamma $  contained in  $D \cap (Y \setminus U)$ with a base point $x$, where 
$D \simeq \Delta$ is a small disk on $Y$ transversal to $Y \setminus U$ in
 a point $s_0$.  

Since however $\pi_1 (Y)$ is abelian, we see that $ {\rm Ker}\, c$ is generated by such elements. 

For each such element  take a small disk
$\tilde{D} \simeq \Delta$ in $X$ such that $f(\tilde{D})
= D$, and let $d$ be the  degree of the finite branched cover $\tilde{D} \ra D$.

The preimage of $\gamma$
in $\tilde{D}$ is a closed circle
$\tilde{\gamma}$ such that $f(\tilde{\gamma}) =
d\gamma$.  Thus ${\rm Ker}\, c/(f_{U})_*({\rm Ker}\, b)$,
is a torsion group. Hence ${\rm Coker}\, b_F$  is a finite abelian
group. The last statement is clear from the fact
that
$\pi_1(X)$, $\pi_1(Y)$ and $A$ are all abelian.

\end{proof}

From \cite{Ca04},
Proposition 2.9  (see also \cite{Ca95} Corollary C,
\cite{Ca02}, Proposition 4.8), we cite the following

\begin{theorem}\label{thmCa} Let $X$ be a compact complex manifold of
dimension $n$ such that
\begin{enumerate}
\item The cohomology
ring $H^*(X, \bZ)$ is isomorphic to  the cohomology ring of the
$n$-dimensional complex torus. 
\item  $H^0(X, d\sO_X) =
n$, i.e., there are exactly $n$ linearly independent $d$-closed holomorphic $1$-forms. 
\end{enumerate}
Then $X$ is biholomorphic to a complex torus.
\end{theorem}

If $X$ is bimeromorphically equivalent to a K\"ahler manifold, our main problem is easily
answered. 

\begin{theorem}\label{thm0} Let $X$ be a compact complex manifold
such that
\begin{enumerate}
\item The cohomology ring $H^*(X, \bZ)$
is isomorphic to  the cohomology ring of the $n$-dimensional complex
torus (for instance,
$X$ is homotopy equivalent to a complex torus of dimension $n$).
\item  $X$ is in the Fujiki class $\sC$, i.e., $X$  is
bimeromorphic to a compact K\"ahler manifold.
\end{enumerate}
Then $X$ is biholomorphic to a complex torus of dimension $n$.
\end{theorem}

\begin{proof}
We apply Theorem \ref{thmCa} to our $X$. The first condition in Theorem \ref{thmCa} holds 
by assumption.     
In
particular, $b_1(X) = 2n$.  As $X$ is in class $\sC$, every
holomorphic form is $d$-closed  and the Hodge
decomposition holds for $X$ (\cite{Fj78}, Corollary 1.7).  Thus the second condition in  Theorem \ref{thmCa}
also holds and an application of  Theorem \ref{thmCa} implies the result.
\end{proof}

A special case of Theorem \ref{thm0} is 

\begin{corollary}\label{cor1} A Moishezon manifold $X$  homotopy
equivalent to a complex torus of dimension
$n$ is biholomorphic to an abelian variety.
\end{corollary}

Recall that a compact complex manifold is said to be a {\it Moishezon}
manifold  if the algebraic dimension is maximal: $a(X) = \dim
X.$ 

\section{Complex torus bundles over a complex torus}
\setcounter{lemma}{0}
\noindent In this section we prove two general results on submersions of special manifolds
(Theorems \ref{thm1}, \ref{thm2}). These results are
used in our proof of our Main Theorem \ref{main}. The crucial
point in both results is that we do  {\it not}
assume the total space $X$ to be K\"ahler.

\begin{theorem}\label{thm1} Let $f : X \to Y$ be a holomorphic submersion  
 with connected fibres between compact (connected) complex manifolds and assume:
\begin{enumerate}
\item 
$X$ has complex dimension $n+m$ and trivial canonical divisor $K_X \equiv 0$;
\item $Y$   has complex dimension  $m$ and also
 $K_Y \equiv 0$;
\item every fiber $X_y$ ($y \in Y$) is K\"ahler;
\item the monodromy action of $\pi_1(Y)$ on $H^n(X_y,
\bZ)$ is trivial.
\end{enumerate} 
Then all the fibres $X_y$ are biholomorphic,  and $f$ is a holomorphic
fibre bundle.
\end{theorem}

\begin{proof} By (4), $R^nf_* \bZ_X$ is
not only locally constant but also {\it globally
constant} on $Y$. Thus,  for the $\bZ_Y$ dual local system, we have
$$(R^n f_* \bZ_X)^{*} \simeq H_n(X_b, \bZ)_f \times Y\, .$$ 

Here $b
\in Y$ is any base point and $H_n(X_b, \bZ)_f$
denotes the free part of $H_n(X_b, \bZ)$. The same abbreviation will be
applied  for other points $y \in Y$. Let
$$\gamma_{1, b}, \ldots , \gamma_{k, b}$$ be a basis of $H_n(X_b, \bZ)_f$ and let
$$\tilde{\gamma}_1, \ldots , \tilde{\gamma}_k$$ be the corresponding
flat basis of $(R^n f_* \bZ_X)^{*}$ over $Y$.

Then the elements $\tilde{\gamma}_{i, y}$ ($1 \le i \le k$) form a free basis of
$H_n(X_y, \bZ)_f$ for each $y \in Y$.

Now, following Fujita \cite{Fu78}, p. 780-781, we construct a
family of holomorphic $n$-forms on the fibres, say
$\{\varphi_y\}_{y \in Y}$, which varies holomorphically with respect
to $y \in Y$. 

Since
$$\omega_{X/Y} = \sO (K_X) \otimes f^* \sO(K_Y)^{\vee} \simeq \sO_X\,\,
,$$
we obtain that $$f_* (\omega_{X/Y} ) \cong \hol_Y. $$
  We are done by the exact sequence 
$$ 0 \ra f^* ( \Omega_Y^1) \ra \Omega_X^1 \ra \Omega_{X|Y}^1 \ra 0$$
and since by definition 
$$ \omega_{X/Y} : = det ( \Omega_{X|Y}^1) = \Lambda^n (  \Omega_{X|Y}^1) .$$

Hence a global generator of $f_* (\omega_{X/Y} ) \cong \hol_Y $ gives the desired family of
holomorphic $n$-forms on the fibres, yielding a nowhere vanishing form on each fibre.

Note that $\varphi_y$ is $d$-closed being a top holomorphic form.

\vskip .2cm \noindent
Now we consider the {\it non-projectivized, global} period map:
$$\tilde{p}_Y : Y \to \bC^k\,\, ;\,\, y \mapsto
(\int_{\tilde{\gamma}_{i, y}} \varphi_y)_{i=1}^k\,\, .$$ This map is
holomorphic by a fundamental result of Griffiths. Indeed, to be able to apply
\cite{Gr68}, Theorem (1.1), we need that the fibres $X_y$ are
K\"ahler  but we do not need that the total space $X$ be
K\"ahler. 

On the other hand, since $Y$ is compact, the global {\it
holomorphic} functions on $Y$ are constant. Thus all functions
$$ y \mapsto \int_{\tilde{\gamma}_{i, y}} \varphi_y$$ 
are constant on $Y$.
Hence the usual period map $p_Y : Y \to \bP^{k-1}$,
which is just the projectivization of the target domain $\bC^k$ of
$\tilde{p}_Y$, is also constant as well.

As all the fibers $X_y$ ($y \in Y$) are compact K\"ahler manifolds with trivial
canonical class, the local Torelli Theorem holds
for them, i.e., the period map from  the Kuranishi space to the
period domain is injective (see e.g. \cite{GHJ03},
p. 109, Theorem 16.9; the proof given there is written only for
Calabi-Yau 3-folds, but the proof in the general case is
exactly the same). 

Since $p_Y$ is constant  and $Y$ is connected, it
follows that all the fibres $X_y$ are biholomorphic.
Hence $f$ is locally analytically trivial by the fundamental result
of Grauert-Fischer (or by Kuranishi's theorem). This concludes the proof.
 
\end{proof}

\begin{theorem}\label{thm2} Let $f : X \to Y$ be a holomorphic submersion  
 with connected fibres between compact (connected) complex manifolds and assume:
\begin{enumerate}
\item $ X$ is  homotopy equivalent to a complex 
torus  of dimension $n+m$; 
\item $Y$ is a complex 
torus of dimension $m$; 
\item some fibre $X_y$  is 
biholomorphic to a complex torus.
\end{enumerate}
Then $f : X \to Y$ is a principal holomorphic torus bundle  and $X$
is biholomorphic to a complex torus.
\end{theorem}

\begin{proof} By \cite{Ca04}, Theorem 2.1, every fibre $X_y$ is
isomorphic to  a complex torus of dimension $n$. Let
$F = X_y$ be one of the fibres of $f$. Since $\pi_2(Y) = 0$, we have the
following exact sequence
$$0 \to \pi_1(F) \simeq \bZ^{2m} \to \pi_1(X) \to \pi_1(Y) \simeq
\bZ^{2n} \to 0\, .$$ 

Since $\pi_1(X) \simeq
\bZ^{2(n+m)}$ by (1), this sequence splits and $\pi_1(Y)$ acts on $\pi_1(F)$
as the identity. Then, by Theorem \ref{thm1}, $f$
is a holomorphic fibre bundle. In particular, the Kodaira-Spencer map
$$T_{Y, y} \to H^1(X_y, T_{X_y})$$ 
of $f$ is zero
at every point $y \in Y$. Then, by \cite{Ca04}, Proposition 3.2 and its
proof, $f$ is a principal fibre bundle with
structure group $F$, i.e., a fibre bundle whose transition
functions are given by translations by local holomorphic
sections of $F$ over $Y$. We want to show that they can actually
chosen to be locally constant.
\vskip .2cm \noindent
To verify this, we follow \cite{BHPV04}, p.196. Set $\Gamma = H_1(F,
\bZ) \simeq \bZ^{2n}$. 

Consider the following
commutative diagram of exact  sequences of abelian sheaves on $Y$:
\[
\begin{CD} 0 @>>> \Gamma = \Gamma_Y @>>> \bC_Y^n @>>> F_Y @>>> 0\\
@VVV @V=VV @VVV @VVV @VVV\\ 0 @>>> \Gamma =
\Gamma_Y @>>> \sO_Y^n @>>> \sF_Y @>>> 0
\end{CD}
\]
\par
\vskip 4pt
\noindent Here $F_Y$ is the abelian sheaf of locally constant
sections with values in $F$  and $\sF_Y$ is the abelian sheaf
of holomorphic sections with values in $F$. 

Taking the 
corresponding cohomology sequences yields the diagram
\[
\begin{CD} H^1(Y, \bC^n) @>\gamma>> H^1(Y, F_Y) @>>> H^2(Y, \Gamma)\\
@V\beta_1VV @V\beta_2VV @V=VV\\ H^1(Y, \sO_Y^n)
@>\alpha>> H^1(Y, \sF_Y) @>c>> H^2(Y, \Gamma)
\end{CD}
\]
\par
\vskip 4pt
\noindent Let $\eta \in H^1(Y, \sF_Y)$ be the class representing the
principal holomorphic bundle structure of
$f : X \to Y$. Set $\epsilon = c(\eta)$.  Note that $f$ is
topologically trivial, since the exact sequence of the
fundamental group splits trivially. Thus $\epsilon = 0$ and therefore $\eta
= \alpha(\eta_1)$  for some $\eta_1 \in H^1(Y,
\sO_Y^n)$. Since $Y$ is K\"ahler,
the map $\beta_1$ is the one induced by the
natural projection under the Hodge  decomposition
$$H^1(Y, \bC) = H^1(\sO_Y) \oplus H^0(\Omega_Y^1)\, .$$ In
particular, $\beta_1$ is surjective. Thus $\eta_1 =
\beta_1(\eta_2)$ for  some $\eta_2 \in H^1(Y, \bC^n)$. Hence
$$\eta = \alpha\beta_1(\eta_2) = \beta_2\gamma(\eta_2) =
\beta_2(\eta_3),$$ where $\eta_3 = \gamma(\eta_2) \in H^1(Y,
F_Y)$. This means that the transition functions defining the
principal bundle structure $f : X \to Y$ can be chosen to be
{\it locally constant}.
\vskip .2cm \noindent

We can now use two arguments, here is the first.

Let $Y = \cup_{i \in I} U_i$ be a sufficiently small open covering of
$Y$  with trivializations
$$\varphi_i : X_{U_i} \simeq F \times U_i\,\, ,$$ such that the
transition functions $\varphi_i^{-1} \circ \varphi_j$ are all
constant  on $U_i \cap U_j$. Let $\tau_Y$ be a standard K\"ahler form
on $Y$ and $\tau_F$ be a standard K\"ahler form
on $F$. 

Set $\tau_i = \tau_Y \vert U_i$. Then $\tilde{\tau}_i :=
\varphi_i^{*}(\tau_i \wedge \tau_F)$ gives a
K\"ahler form on $X_{U_i}$. As $\varphi_i^{-1} \circ \varphi_j$ is a
translation by some {\it constant} element of $F$ over
$U_i \cap U_j$, it follows that $\tilde{\tau}_i = \tilde{\tau}_j$ on
$X_{U_i} \cap X_{U_j}$. Hence
$\{\tilde{\tau}_i\}_{i \in I}$ defines a  global K\"ahler form on
$X$. In particular
$X$ is K\"ahler and therefore $X$ is biholomorphic to a complex torus  by
Theorem \ref{thm0}.

Alternatively, one immediately sees that the fact that the transition functions defining the
principal bundle structure $f : X \to Y$ are
 locally constant implies that the space of closed holomorphic 1-forms
$ H^0 (X, d \hol_X)  $ has dimension at least $n$, and we can
apply Theorem \ref{thmCa}.

\end{proof}

\section{A characterization of complex tori - the case fibred by curves}
\setcounter{lemma}{0}
\noindent The goal of this section is to prove the following
\begin{theorem}\label{thm3}  Let $X$ be a compact complex manifold
subject to the following conditions.
\begin{enumerate} 
\item $X$ is homotopy equivalent to a
complex torus of dimension $m+1$;
\item there is a dominant
meromorphic map $f : X \rh Y$ to a compact complex manifold
$Y$ with
$\dim\, Y = m$;
\item $ m \leq 2$ or $Y$ is Moishezon with $\kappa (Y) \geq 0;$  
\item $K_X \equiv 0$.
\end{enumerate}
Then $X$ is biholomorphic to a complex torus of dimension $m+1$.
\end{theorem}
In the rest of this section, we shall prove Theorem \ref{thm3} and always assume
the situation of Theorem \ref{thm3}. 
Take a resolution of indeterminacies
$\nu : \tilde{X} \lra X$ of $f$, yielding a surjective morphism
$$\tilde{f} : \tilde{X} \lra Y$$ 
By considering the Stein factorization we may assume that $\tilde f$ 
has connected fibers; loosely speaking $f$ has connected fibers. 
In case $Y$ is Moishezon, we may replace $Y$ by a suitable birational model
and therefore may assume $Y$ to be projective. Finally
$F$ will always denote a smooth fibre of $\tilde{f}$.

\begin{lemma}\label{lem1}
\begin{enumerate}
\item If $\kappa (Y) \geq 0,$ all smooth fibers of $\tilde f$
are isomorphic to a single elliptic curve, say $E$, and $\kappa (Y) = 0.$  
\item If moreover $Y$ is projective, then
$Y$ is birational to an abelian variety of dimension $m$.
More                                               
precisely, the  Albanese  map $a : Y \to {\rm
Alb}\, Y$ is a birational surjective morphism.  
\end{enumerate}                      
\end{lemma}
                                                                  
\begin{proof} (1) Since $K_{\tilde X}$ is effective, the fiber $F$ has genus 
$g(F) \ge 1$. Then by \cite{Ue87}, Theorem 2.1, $F$ must actually be
an elliptic curve. Moreover by  \cite{Ue87}, Theorem 2.2, 
$$0 = \kappa(X) = \kappa(\tilde{X})\, \ge\, {\rm max}\, (\kappa(Y),
{\rm var}(\tilde{f})) \ge 0\, ,$$ 
where ${\rm var}(\tilde{f}) $ denotes the variation of $\tilde f.$      
Thus $\kappa(Y)
= 0$ and ${\rm var}(f) = 0$ and the first assertion is proven. 
\vskip .2cm \noindent
(2) For the second assertion assume now that $Y$ is projective.
Note that
$\pi_1(\tilde{X}) \simeq \bZ^{2(m+1)}$, since 
$X$ is homotopy equivalent to  a complex torus of dimension $m+1. $ Thus, applying Proposition \ref{prop1},
$$\pi_1(Y) \simeq \bZ^n$$ (up to torsion) for some integer
$n$ such that $2m \le n \le 2(m+1)$. Since $Y$ is a projective manifold,
Hodge decomposition gives 
either $n = 2m$ or $n = 2(m+1)$ and $h^1(\sO_Y) = m$ or $h^1(\sO_Y) = m+1$.   Since 
$\kappa(Y) = 0$, a fundamental result due to
Kawamata (\cite{Ka81}, main theorem) yields 
$$h^1(\sO_Y) =
m ( = \dim \, Y,)$$ and also the birationality of the  Albanese morphism $a : Y \lra {\rm Alb} \ Y.$
This completes the proof.
\end{proof}

\begin{lemma}\label{lem1'} Assume that $m \le 2$. Then $X$ is a complex torus or the following two statements
hold (recall that we assume $f$ to have connected fibers).
\begin{enumerate}
\item All smooth fibres are isomorphic to a fixed elliptic curve, say $E$;
\item $Y$ is bimeromorphic to a complex torus of dimension $2$.  More
precisely, the  Albanese  map $a : Y \to {\rm
Alb}\, Y$ is a bimeromorphic surjective morphism.
\end{enumerate} 
\end{lemma}

\begin{proof} When $m=1$, we have $\dim\, X =2$. Then by classification $X$ is a
complex torus since $K_X = \sO_X$ and $b_1(X) = 4$. 
So from now we shall assume that $m=2$. We may also assume that $Y$ is a minimal surface.

Suppose first that $\kappa(Y) \ge 0$, hence $\kappa (Y) = 0$ by 4.2(1), and $F$ is a fixed elliptic curve. If $Y$ would not be
complex torus, then $b_1(Y) \le 3$ by classification. Then however 
$$b_1(F) + b_1(Y) \le 5 < 6 = b_1(X) = b_1(\tilde{X})\,\, ,$$ a
contradiction to Proposition
\ref{prop1}.

It remains to consider the case where $\kappa(Y) = -\infty$.
If in addition $Y$ is  K\"ahler, then $Y$ is projective (rational or
birationally ruled)  by classification. We also
have $b_1(Y) \le 2$  by the fact that $\pi_1(Y)$
is abelian (Proposition \ref{prop1}). Then $b_1(F) \ge 4$ for the
general  fibre $F$ of $\tilde{f} : \tilde{X} \to
Y$, 
again by Proposition \ref{prop1}. Therefore $g(F) \ge
2$,  where $g(F)$ is the genus of the curve $F$. Then we have a
relative pluri-canonical map
$\hat{X} \rh Z$ of $\tilde{X}$ over $Y$ (\cite{Ue75}, Theorem 12.1
and its proof). As $Y$ is projective, $Z$ is a
projective $3$-fold by the construction given there. Hence
$$a(X) = a(\tilde{X}) = a(Z) = 3$$  and we conclude that $X$ is
biholomorphic to a complex torus  by Theorem
\ref{thm0} or by Corollary \ref{cor1}, and we  are done.

If $\kappa(Y) = -\infty$ and $Y$ is not K\"ahler, then $Y$ is a
minimal surface of class $VII$. In particular,
$Y$ is not covered by rational curves and $b_1(Y) = 1$. Now observe that $f$ is almost holomorphic
in the sense that $f$ is proper
holomorphic over Zariski dense  open subset of $Y$. Indeed, otherwise
the exceptional locus  of the resolution of indeterminacies $\tilde{X} \to X$ 
dominates $Y$, so that $Y$ would be dominated by a uniruled surface contradicting the assumption that $Y$ is of class $VII.$
Now $f$ being almost holomorphic, the
general fibre $F$ of $\tilde{f}$ is an elliptic curve by adjunction.
Thus
$$b_1(F) + b_1(Y) = 3 < 6 = b_1(X) = b_1(\tilde{X}),$$  a
contradiction to Proposition \ref{prop1}. This completes the
proof.
\end{proof}

The upshot of the preceeding two lemmata is that we may assume $Y$ to be a torus. 
In particular the meromorphic  map $f : X \rh Y$ (from our {\it original} $X$) is
holomorphic and all smooth fibers 
are isomorphic to a fixed elliptic curve
$E$.

\begin{lemma}\label{lem2}
\begin{enumerate}
\item  $f$ is smooth in codimension $1$, that is, the set  of critical
values of $f$ is of codimension $\ge 2$ on $Y$;
\item $f$ is equi-dimensional, or equivalently, $f$ is a flat morphism.
\end{enumerate} 
\end{lemma}

\begin{proof} 
(1) (a) Let us first consider the case that $Y$ is projective. 
Then we take a general complete
intersection curve $C$  on $Y$, i.e., a complete
intersection of $m-1$ general hyperplanes  of $Y$. So by Bertini's
theorem, $C$ is a smooth curve and
$X_C = f^{-1}(C)$  is a smooth surface. Let $f_C : X_C \lra C$ be the
induced  morphism; then it suffices to show
that $f_C$ is a smooth morphism.  By the adjunction formula, by
$K_X = \sO_X$ and $K_Y = \sO_Y,$ we obtain
$$K_{X_C} = f_C^*(K_C), $$
i.e., $K_{{X_C}/C} = \sO_{X_C}$. Then the canonical
bundle formula for an elliptic surface (see eg.
\cite{BHPV04}, Page 213,  Theorem 12.3) gives the
smoothness of $f_C$.

(b) It remains to consider the case where $\dim\, Y = 2$ and $Y$ not projective. If $a(Y) = 0$, then
$Y$ has no complete curve and $f$ is smooth in
codimension $1$. If $a(Y) = 1$, then the algebraic reduction $a : Y \to C$ of $Y$
is a smooth elliptic fibration over an elliptic
curve $C$ and all curves on $Y$ are fibres of $a$. Thus  the 1-dimensional part of the critical values
form a normal crossing divisor and we can apply the canonical
bundle formula (\cite{Ue87},  Theorem 2.4, or
\cite{Fu86}, Theorem 2.15) to our elliptic $3$-fold
$f : X \to Y$. As a result, if the set of the critical values is {\it
not} of codimension $\ge 2$, then  there are
fibres $C_i$ ($1 \le i \le k$) of $a$ and positive integers $n_i$
and $M$ such that we have a bijection
$$\vert MK_X \vert \leftrightarrow
\vert f^*(MK_Y + \sum_{i=1}^{k} n_iC_i) \vert\, .$$ This is however
absurd, because the left hand side is an empty
set by
$K_X = \sO_X$, but the right hand side is a non-empty set since $K_Y = \sO_Y$ and
$n_i >0$.

This completes the proof of (1).

\vskip .2cm (2) To begin with, notice that equi-dimensionality and
flatness are equivalent, $X$ and $Y$ being smooth. We denote the union of all irreducible components of
dimension $\ge 2$ in the fibers of $f$ by $N_0.$ Assuming
$N_0 \ne \emptyset$
shall derive a
contradiction.
To do that, let
$$N = f^{-1}f(N_0). $$

First of all, $N$ must be of pure
codimension $1$ in $X$. In fact, otherwise take a
general small $m$-dimensional disk $\Delta$ centered  at a general
point $P$ of a 1-dimensional component of $N$. 
Then $\Delta $ dominates $Y$ at $f(P)$ and
$f\vert \Delta :  \Delta \to Y$ is a  generically
finite surjective morphism around $f(P)$ branched in
codimension $\ge 2$ on
$\Delta$. However this is impossible by the purity of the branch loci.
Thus $N$ is a divisor. \\
Choose an irreducible component $B$ of $N.$ \\
By Hironaka's flattening theorem (\cite{Hi75}, main result),
there is a successive sequence of blow-ups $\mu : \hat{Y} \to
Y$ such that  the induced morphism 
$$f_1 : X_1 :=  X \times_{Y} \hat{Y} \to \hat{Y}$$  is a flat
morphism. 
Let 
$$E'_i \  (1 \le i \le k) $$ 
be the exceptional divisors of  
$\mu : \hat{Y} \to Y$. Since flatness is preserved under base
change, we may assume that
$\sum_{i=1}^k E'_i$ is a normal crossing divisor, possibly performing
further blow-ups of $\hat{Y}$. Consider the
normalization $X_2 \to X_1$ of $X_1$ and perform a
resolution of singularities (\cite{Hi77}, main result)
of $X_2$, say $X_3 \to X_2$  and then finally take a 
resolution of indeterminacies (ibid.) of
$X \rh X_3$, say
$\pi : \hat{X} \to X$.  Let $\hat{f} : \hat{X} \to \hat{Y}$ be the
induced morphism.

Let $E_j$ ($1 \le j \le \ell$) be the exceptional divisors  of $\pi :
\hat{X} \to X$ and $\hat{B}$ be the proper
transform  of $B$ on $\hat{X}$. Since $B$ is of codimension $1$ on $X$, necessarily
$\hat{B} \not= E_j$ for any $j$. On the other hand, the fact that we have flattened $f$ means that$\hat{f}(\hat{B})$ is one of the
$E'_i$,  say $E'_1$.

We are going to apply the canonical bundle formula for $\hat{f}$ in
\cite{Ue87},  Theorem 2.4 (or \cite{Fu86}, Theorem 2.15).
Note that 
$$K_{\hat{X}} = \sum_{j=1}^{\ell} a_jE_j$$ 
with every $a_j >
0$ since $K_X = \sO_X$. For the same reason, 
$$K_{\hat{Y}} =
\sum_{i=1}^{k} b_iE'_i$$ 
with every $b_i > 0$. 
As $f$ is smooth in
codimension $1$, the discriminant divisor of
$\hat{f}$ is supported  in $\bigcup_{i=1}^{k} E'_k$. Thus for a large
multiple $M >0$, we obtain
$$M\sum_{j=1}^{\ell} a_jE_j = MK_{\hat{X}}$$
$$= \hat{f}^{*}(MK_{\hat{Y}} + \sum_{i=1}^{k} c_iE'_i) + D_1 - D_2  =
\hat{f}^{*}(\sum_{i=1}^{k} ((b_i + c_i) E'_i) + D_1
- D_2$$  where $D_1$ is an effective divisor such that no multiple of $D_1$ moves, 
 $D_2$ is an effective
divisor such that $\hat{f}(D_2)$  is of
codimension $\ge 2$, and each $c_i$ is a non-negative integer.
Notice $b_i + c_i > 0$ for all $i$.  Moreover, by
\cite{Ue87},  Theorem 2.4 (especially the statement (6) there), every element of
$\vert MK_X \vert$ is uniquely written, as a divisor, in the form  of
the sum of an element of
$$\vert \hat{f}^{*}(\sum_{i=1}^{k} (b_i + c_i) E'_i) \vert = \hat{f}^* \vert
\sum_{i=1}^{k} (b_i + c_i) E'_i \vert =
\{ \hat{f}^{*}(\sum_{i=1}^{k} (b_i + c_i) E'_i) \}$$  and the divisor
$D_1 - D_2$. Thus, we have the following equality
as a {\it divisor}:
$$M\sum_{j=1}^{\ell} a_jE_j + D_2 =
\hat{f}^{*}(\sum_{i=1}^{k} (b_i + c_i) E'_i) + D_1\,\, .$$ In this
equality of divisors, the prime divisor $\hat{B}$
appears  in the right hand side because $\hat{f}(\hat{B}) = E'_1$. But it
does not  appear in the left hand side, since - as already observed - 
$\hat{B}$ is neither $\pi$-exceptional  nor $\hat{f}(\hat{B})$ is of
codimension $\ge 2$ in $\hat{Y}$.  This
contradiction concludes the equi-dimensionality of $f$ and 
the flatness of $f$ as well.
\end{proof}

\begin{lemma}\label{lem3} The map $f: X \to Y$ is smooth.
\end{lemma}

\begin{proof} Let $y \in Y$ be any point of $Y$ and suppose that the fiber $X_y$ is singular.
If $X_y$ has a non-reduced component $C$, necessarily of dimension 1 by Lemma \ref{lem2}, 
we choose  an $m$-dimensional general disk
$\Delta$ centered at a general non-reduced  point $P \in C$. Then $f \vert \Delta : \Delta \to Y$  is a generically
finite surjective morphism around $f(P)$ whose branch
locus in $\Delta$ is of codimension $\ge 2$ (since $f$ is smooth in
codimension $1$),  a contradiction to the purity of
branch loci. \\
Hence $X_y$ is reduced.
Now take a local section
$D$ at a general point of $X_y$. Once we have chosen $D$, we can describe the fibration
by the Weierstrass equation locally near $y$:
$$y^2 = x^3 +a(t)x + b(t)$$ where $a(t)$, $b(t)$ are holomorphic
functions around $y$.  Then the critical locus of
the original $f$ around $y$ is given by the equation$$4a(t)^3 +
27b(t)^2 =0\, .$$ In particular, it is of pure
codimension $1$ on $Y$ unless it is empty.  As $f$ is smooth in
codimension $1$, it follows that the critical locus
is empty, that is, $f$ is smooth.
\end{proof} Finally we can apply Theorem \ref{thm2} to conclude Theorem
\ref{thm3}.

\hfill Q.E.D. for Theorem \ref{thm3}.

\begin{remark} {\rm   In case $\dim X = 3,$ Lemma \ref{lem3} can be proved without using
 Weierstrass models in the following way.
So we suppose that we already as shown in the first part of the proof of Lemma 4.5
 without using the Weierstrass normal form that
$f$ can has finitely many singularities, say $x_1, \ldots, x_N.$ 
Choose a general holomorphic 1-form  
$\omega$ on $Y$. Then $f^*(\omega)$ vanishes exactly at $x_1, \ldots, x_N$, 
hence $c_3(\Omega^1_X) > 0.$ But $c_3(X) = \chi_{\rm top}(X) = 0,$   
since $X$ is homotopy equivalent to a torus. }  
\end{remark}

\par
\vskip 4pt
\section{A characterization of complex tori -  the case fibred over a curve}
\setcounter{lemma}{0}
\noindent In this section we shall prove the following:
\begin{theorem}\label{thm4} Let $X$ be a compact complex manifold
such that
\begin{enumerate}
\item $X$ is homotopy equivalent to a
complex torus of dimension $3$;
\item  there is a dominant
meromorphic map $f : X \rh Y$ to a smooth compact curve;
\item  $\sO_X(K_X ) \cong  \sO_X$.
\end{enumerate}
Then $X$ is biholomorphic to a complex torus.
\end{theorem}

Observe that in condition (2), taking the Stein factorization, 
we may assume $f$ to have connected fibers. \\

We  start with
\begin{lemma}\label{lem4}  Let $X$ be a compact complex manifold subject to the
assumptions in  Theorem \ref{thm4}. Then $X$ is biholomorphic to  a complex torus of
dimension $3$  or $Y$ is an elliptic curve, $f$ is a
holomorphic map and the general smooth fibre of $f$ is a
complex torus of dimension $2$.
\end{lemma}

\begin{proof} By Proposition \ref{prop1}, $\pi_1(Y)$ is an abelian
group.  Hence $Y$ is either an elliptic curve or
$\bP^1$.

Consider first the case that $Y$ is an elliptic curve, so $f$ is holomorphic. Let $F$  be a general fibre of $f$. Then $K_F = \sO_F$ by the adjunction formula.  Since 
$$b_1(F) + b_1(Y)
\ge b_1(X) = 6 $$ 
by Proposition \ref{prop1}, it follows from the
classification  of compact complex surfaces with
$\kappa = 0$ 
(see e.g. \cite{BHPV04}, p.244, Table 10)  
that $b_1(F)
= 4$  and $F$ is a complex torus of dimension
$2$ so that we are done. \\
In case $Y = \bP^1$ let $\tilde{f} : \tilde{X} \to
Y$  be a resolution of indeterminacies of $f$
and let $F$ be a  general fibre  of $\tilde{f}$. Then $F$
is smooth and $\kappa(F) \ge 0$ by the
adjunction formula. \\
If $\kappa(F) \ge 1$, then we can take a relative pluri-canonical map
$\varphi : X \rh Z$ from $X$ over $Y$ (\cite{Ue75}, Theorem 12.1 and
its proof). As $Y$ is projective and $Z$ is
projective over
$Y$, it follows that $Z$ is projective. We have also $\dim\, Z \ge
2$. Thus $X$ is biholomorphic to a complex torus by
Theorem \ref{thm3}. \\
If $\kappa(F) = 0$, then $b_1(F) \le 4$ again by classification  of
compact complex surfaces with $\kappa = 0$. 
Then however
$$b_1(F) + b_1(Y) \le 4 < 6 = b_1(X) = b_1(\tilde{X})\, ,$$
contradicting Proposition \ref{prop1}. \\
This completes the proof.
\end{proof}

>From now on we may assume that we have a surjective holomorphic map $f : X \to Y$
over an elliptic curve $Y$ with connected fibres.

The next two propositions of more topological nature will be
applicable to many other situations:

\begin{proposition}\label{prop4} 
Let $X$ resp. $Y$ be a topological space  which is
homotopy equivalent to  a real torus $A$ of real dimension $N$ resp. homotopy
equivalent to a real torus $B$ of real dimension $r$.
Let $f : X \to Y$ be a continuous
map which
is dominant in the sense that $f^* : H^r(Y, \bZ)
\to H^r(X, \bZ)$ is non-zero. Let $u : \hat{Y} \to Y$ be  the             
universal covering map and $\hat{X} = X \times_{Y}
\hat{Y}$ be the fibre  product. Then $\hat{X}$ is homotopy equivalent to
a real torus of real dimension $N-r$, in particular
we have $H^{N-r}(\hat{X}, \bZ) = \bZ$.
\end{proposition}

\begin{proof} Notice that we have  natural isomorphisms
$\pi_1(X) \simeq H_1(X, \bZ)$ and
$\pi_1(Y) \simeq H_1(Y, \bZ)$ and - by our assumptions - they are 
isomorphic as abstract groups to $\bZ^N$ and $\bZ^r$, respectively. 
Let us
consider the homomorphism
$$f_  : \pi_1(X) \simeq \bZ^N \to \pi_1(Y) \simeq \bZ^r\,\, ,$$
induced by $f$. Under the  above isomorphisms, this
homomorphism is  the same as the homomorphism
$$f_* : H_1(X, \bZ) \simeq \bZ^N \to H_1(Y, \bZ) \simeq \bZ^r\,\, .$$
The dual homomorphism 
$$(f_1)^{*} : H^1(Y, \bZ) \simeq \bZ^r \to H^1(X, \bZ) \simeq
\bZ^N\,\, $$ is a part of the homomorphism of algebras given by the
pullback
$$f^{*} : H^*(Y, \bZ) \to H^*(X, \bZ). $$
Again since $X$ and $Y$ are homotopy equivalent to
real tori of dimensions $N$ and $r$ respectively,
we know that
$$\wedge^{r} H^1(Y, \bZ) \simeq H^{r}(Y, \bZ) \simeq \bZ\,\, ,\,\,
\wedge^{r} H^1(X, \bZ) \simeq H^{r}(X, \bZ)\,\, ,$$ and the natural
homomorphism
$$(f_r)^* : H^r(Y, \bZ) \simeq \bZ \to H^r(X, \bZ)$$  is simply
$\wedge^{r} (f_1)^{*}$. As this is not zero by
assumption,  it follows that
$(f_1)^{*}$ is injective. Hence $f_* :\pi_1(X) \to \pi_1(Y)$ is
surjective up to a finite cokernel.

Now $f$ factors through the  finite unramified cover $Y'$ of $Y$ corresponding
to $f_*(\pi_1(X))$, whence we may  replace $Y$ by $Y'$ and
assume that $f_*$ is surjective.

We have then
${\rm Ker}\, f_* \simeq \bZ^{N-r}$ and $\pi_1(X)$ splits as
$$\pi_1(X) \simeq \bZ^N \simeq \bZ^r \oplus \bZ^{N-r}
\simeq \pi_1(Y) \oplus {\rm Ker}\, f_{*}\,\, .$$  Hence under the
universal covering maps $u_B : \bR^r \to B$
(corresponding to
$\pi_1(Y) \simeq \pi_1(B)$) and $u_A : \bR^N \to A$ (corresponding to
$\pi_1(X) \simeq \pi_1(A)$), it follows that $X'$ is homotopy equivalent  to
$$\bR^{N}/{\rm Ker}\, f_{*} \simeq
\bR^{r} \times (\bR^{N-r}/\bZ^{N-r}) \simeq \bR^{r} \times
T^{N-r}\,\, ,$$ where $T^{N-r}$ is a 
real torus of dimension $N-r.$ The result is now obvious.
\end{proof}

\begin{proposition}\label{prop5}  Let $X$ be a compact complex
manifold homotopy equivalent to a  complex torus of
dimension $n+1$ and let $Y$ be an elliptic curve and let $f : X \to Y$ be a
surjective holomorphic  map with connected
fibres. \\
Then all analytic sets $f^{-1}(y)$  ($y \in Y$) are irreducible.
Moreover,  if $Z$ is the
reduction of a singular fiber and $\tilde{Z}$ is a resolution of
singularities  of $Z$, then there is a real torus of
real dimension
$2n$ and a surjective differentiable map
$\rho : \tilde{Z} \to T^{2n}$ such that the induced homomorphism
$$\rho_* : \pi_1(\tilde{Z}) \to \pi_1(T^{2n})$$ is surjective.
\end{proposition}

\begin{proof} As $f$ is proper holomorphic and $Y$ is compact,  the
critical values of $f$ consist of finitely many
points, say
$$B := \{b_1, b_2, \ldots , b_k \}\,\, .$$ Let $u : \bC \to Y$ be the
universal cover of $Y$. We consider the fiber product
$\hat{X} = X \times_{Y} \bC$ and let $\hat{f} : \hat{X} \to \bC$ be the
induced holomorphic map. The set of critical values of
$\hat{f}$  is $\hat{B} = u^{-1}(B)$. This  is a discrete
set of points of $\bC$. By an appropriate choice of the origin in
$\bC$, we may assume that $0 \not\in \hat{B}$. Then $u^{-1}(u(0))$
forms  a lattice $\Lambda$ such that $Y =
\bC/\Lambda$. We choose generators  of $\Lambda$, say $v_1$ and
$v_2$. Then each region
$$U_{n, m} := \{\alpha v_1 + \beta v_2 \, \vert\, n \le \alpha < n+1\, ,
m \le \beta < m+1
\}$$  ($n, m \in \bZ$) forms a fundamental domain for
$Y$. We first take the following  contractible graph
$$\Gamma_0 := \bR v_1 \cup \bigcup_{k \in \bZ} (\bR v_2 + kv_1)$$  in $\bC$.
Then connect each $nv_1 + mv_2 \in \Gamma_0$  to each
point  $b_{k, n,m}$,  of $\hat{B} \cap U_{n, m}$ by a   simple path, say
$\gamma_{k, n.m}$ in $U_{n, m}$ so that they are mutually
disjoint. Then
$$\Gamma := \Gamma_0 \cup \bigcup_{k, n, m} \gamma_{k,n, m}$$ becomes a
contractible tree connecting $0$ with the end
points, which are the critical values  of $\hat{f}$.

We next remove from
$\Gamma$ all the end points $b_{k, n, m}$ and
denote the resulting space by
$$\Gamma' := \Gamma \setminus \bigcup_{k, n, m} \{b_{k,n,m}\}\,\, .$$

Finally for each of the removed end points
$b_{k,n,m}$ we fill in a small ball
$B_{k,n,m}$ centered at $b_{k,n,m}$ and denote the resulting space by
$$\tilde{\Gamma} := \Gamma \cup \bigcup_{k, n, m} B_{k,n,m} = \Gamma'
\cup \bigcup_{k, n, m} B_{k,n,m}\,\,.$$

  We put
$$\hat{X}_{\Gamma} = \hat{f}^{-1}(\Gamma)\,\, ,\,\,
\hat{X}_{\Gamma'} = \hat{f}^{-1}(\Gamma')\,\, ,
\hat{X}_{k,n,m} = \hat{f}^{-1}(B_{k,n,m})\,\, ,\,\,$$
$$\hat{X}_{\tilde{\Gamma}}  = \hat{f}^{-1}(\tilde{\Gamma})\,\, ,\,\,
F_{k,n,m} = \hat{f}^{-1}(b_{k,n,m})\,\, .$$  As
$\Gamma$ is tree, one can choose a neighbourhood $U \subset \bC$  of
$\Gamma$ such that $\Gamma$ is a deformation
retract of $U$  and $U$ is also a deformation retract of $\bC$. Then
we obtain a  deformation retract from $\hat{X}$
to $\hat{X}_{\tilde{\Gamma}}$ and then to
$\hat{X}_{\Gamma}$.

Combining this with Proposition \ref{prop4}, we obtain
$$H^{2n}(\hat{X}_{\tilde{\Gamma}}, \bZ) \simeq H^{2n}(\hat{X}, \bZ)
\simeq \bZ\, .$$ On the other hand,
$H^{2n}(\hat{X}_{\tilde{\Gamma}}, \bZ)$ can be also  computed as
follows. We notice that $\Gamma'$ is contractible and that
the fibre over $\Gamma'$ is smooth  and homeomorphic to $F_0$. Here
$F_0$ is the fibre over  the base point $0$. Thus
$$H^{*}(\hat{X}_{\tilde{\Gamma}'}, \bZ) \simeq H^{*}(F_0, \bZ)\,\,
.$$ As each $F_{k,n,m}$ is a deformation retract
of $\hat{X}_{k,n,m}$  and since all the $\hat{X}_{k,n,m}$'s  are mutually disjoint
($B_{k,n,m}$ being sufficiently small),  it follows that
$$H^{*}(\bigcup_{k,n,m} \hat{X}_{k,n,m}, \bZ) \simeq \oplus_{k,n,m}
H^{*}(F_{k,n,m}, \bZ)\,\, .$$

Moreover, since each
$B_{k,n,m} \cap \Gamma'$ is  contractible and since the fibres over this set
are homeomorphic to $F_0$, we also have
$$H^{*}(\hat{X}_{\tilde{\Gamma}'} \cap (\bigcup_{k,n,m} \hat{X}_{k,n,m}), \bZ)
\simeq \bigoplus_{k,n,m} H^{*}(F_0, \bZ)\,\, .$$  
Thus by the
Mayer-Vietoris exact sequence we obtain
$$H^{2n}(\hat{X}_{\tilde{\Gamma}}, \bZ) =
H^{2n}(\hat{X}_{\tilde{\Gamma}'} \cup (\bigcup_{k,n,m} \hat{X}_{k,n,m}),
\bZ)$$
$$\simeq (H^{2n}(F_0, \bZ) \oplus \bigoplus_{k,n,m}  H^{2n}(F_{k,n,m},
\bZ))/(\bigoplus_{k,n,m} H^{2n}(F_0, \bZ))$$
$$\simeq H^{2n}(F_0, \bZ) \oplus \bigoplus_{k,n,m}  (H^{2n}(F_{k,n,m},
\bZ)/H^{2n}(F_0, \bZ))\,\, .$$ 
Since
$H^{2n}(\hat{X}_{\tilde{\Gamma}}, \bZ) \simeq \bZ$, it follows that
$$H^{2n}(F_{k,n,m}, \bZ) \simeq H^{2n}(F_0, \bZ)$$  for each singular
fibre $F_{k,n,m}$. Since $H^{2n}(F_0, \bZ) \simeq
\bZ$,  it follows that
$$H^{2n}(F_{k,n,m}, \bZ) \simeq \bZ\,\, .$$  This implies the
irreducibility of $F_{k,n,m}$, because
$F_{k,n,m}$ is a compact connected complex space of pure dimension
$n$ (hence of real dimension $2n$) so that the rank  of
$H^{2n}(F_{k,n,m}, \bZ)$ is the cardinality of  the set of
irreducible components of $F_{k,n,m}$.

Let $y \in Y$ and $Z = (X_y)_{\rm red}$ be the reduction of  the
fibre $X_y$ of the original fibration $f$.  Now we
know that $Z$ is irreducible. Moreover, by the proof  of Proposition
\ref{prop1}, we also know that the image of  the
natural map
$$\pi_1(Z) \to {\rm Ker}\, (\pi_1(X) \to \pi_1(Y))
\simeq \bZ^{2n}$$ has finite cokernel. Thus the image is
isomorphic  to $\bZ^{2n}$ as well. Consequently we have a
surjective homomorphism
$$\pi_1(Z) \to \bZ^{2n}\,\, .$$ Since $\bZ^{2n}$ is isomorphic to the
fundamental group of a real torus of dimension
$2n$, say $T^{2n}$,  the surjective morphism above is induced  by a
dominant continuous map
$$a : \pi_1(Z) \to \pi_1(T^{2n})\,\, .$$

Since $\pi_1(T^{2n})$ is
commutative, this naturally induces a surjective
morphism $$a : H_1(Y, \bZ) \to H_1(T^{2n}, \bZ)\,\, .$$ Passing to
the dual, we obtain an injective morphism
$$a^{*} : H^1(T^{2n}, \bZ) \simeq \bZ^{2n} \to H^1(Z, \bZ)\,\, .$$

Let $\nu : \tilde{Z} \to Z$ be a resolution  of singularities of the
complex space $Z$ (\cite{Hi77}, main result).
Then the composition of $a$ and $\nu$ defines a continuous map  $\tilde{a}$
such that its action on the first homology is the composition $\tilde{a}_*$
of $\nu_* : H_1(\tilde{Z}, \bZ) \to H_1(Z, \bZ)$ with $a_* : H_1(Z, 
\bZ) \to H_1(T^{2n}, \bZ)\,\, .$

Let $\langle \varphi_i \rangle_{i=1}^{2n}$ be a
basis of $H^1(T^{2n}, \bZ)$ and consider their inverse images
$ \tilde{a}^* (\varphi_i)$ as being represented by $d$-closed differential
forms. Then the  map given by
integration
$$\tilde{Z}\ni x \mapsto (\int_{x_0}^{x} \tilde{a}^*(\varphi_i))_{i=1}^{2n}$$
gives a differentiable map $\rho : \tilde{Z} \to T^{2n}$
such that the induced morphism $\rho_* : H_1(\tilde{Z}, \bZ) \to 
H_1(T^{2n}, \bZ)$
is the homomorphism  $\tilde{a}_*$.

Since we have  isomorphisms $$H^{2n}(\tilde{Z}, \bZ) \cong H^{2n}(Z, \bZ)
\cong H^{2n} (T^{2n}, \bZ)$$ it follows that  $\tilde{a}$ is dominant and
that we have  a surjective
homomorphism $\pi_1(\tilde{Z}) \to \bZ^{2n}$.

\end{proof}

Let us return to the proof of Theorem \ref{thm4} and recall that by Lemma \ref{lem4} we may assume $Y$ to be an elliptic curve. We need
only to show that $f$ is smooth; then  Theorem \ref{thm4} follows from
Theorem \ref{thm2}. \\
So assume that  a fibre $Z$ of $f$ is
singular. We already know that $Z$ is irreducible by  Proposition
\ref{prop4}.  Denote by $m$ the multiplicity of
$Z$ so that
$Z = mZ_{\rm red}$. Since $X$ is smooth, hence
$Z_{\rm red}$ is also a Cartier divisor on $X$. In particular,  the
dualizing sheaf $\omega_{Z_{\rm red}}$ is
invertible.  More precisely, by the adjunction formula and by $K_X =
\sO_X$,  we have
$$\omega_{Z_{\rm red}} = \sO_X(Z_{\rm red}) \otimes \sO_{Z_{\rm
red}}$$ and therefore
$$\omega_{Z_{\rm red}}^{\otimes m} \simeq \sO_{Z_{\rm red}}\, .$$  Since
$K_X = \sO_X$, the multiplicity $m$ is nothing but
the minimal positive integer satisfying this isomorphism (see eg.
\cite{BHPV04}, p.111, Lemma 8.3).

Let $\tilde{Z}$
be the minimal  resolution of the normalization $Z'$ of $Z_{\rm
red}$. Since  $Z_{\rm red}$ is Gorenstein, the conductor ideal of $Z' \to 
Z_{\rm red}$  is
of pure dimension $1$ (if $Z_{\rm red}$ is not normal). Moreover,
since $\tilde{Z}$ is a minimal resolution, the canonical divisors of 
$Z_{\rm red}$ and $\tilde{Z}$
differ by an effective divisor, called classically the 
subadjunction divisor.
We conclude the well known formula
$$\omega_{\tilde{Z}}^{\otimes m} \simeq \sO_{\tilde{Z}}(-D)\,\, ,$$
where $D$ is an effective divisor, possibly $0$ (if and only if $Z_{\rm red}$ 
is normal with at most rational double points as singularities). 
\vskip .2cm Suppose first $D = 0$, hence $\kappa(\tilde{Z}) = 0$. Since $\pi_1(\tilde{Z})$  maps
onto $\pi_1(T) \simeq \bZ^4$ by Proposition
\ref{prop4}, it follows again from  surface classification

  that $\tilde{Z}$ is a complex torus of dimension
$2$.  Since a complex torus of dimension $2$ has no curve with 
negative self-intersection,
it has no non-trivial crepant contraction to a normal
surface $Z'$ . Hence the three surfaces
$$\tilde{Z}\,\, ,\,\, Z'\,\, ,\,\, Z_{\rm red}$$  are all isomorphic.

In particular, $Z_{\rm red}$ is  also a smooth
complex torus (of dimension $2$) and $\omega_{Z_{\rm red}} \simeq
\sO_{Z_{\rm red}}$.  This implies $m=1$ and  $Z
= Z_{\rm red}$ is smooth.
\vskip .2cm
If $D \not= 0$, then $\kappa(\tilde{Z}) = -\infty$. From the
classification of compact complex surfaces with
$\kappa = -\infty$ (see eg. \cite{BHPV04}, p.244, Table 10),
$\tilde{Z}$ is either birationally ruled, say over a curve $C$,
or a surface of class $VII$. If it is birationally ruled, then $H^1(\tilde{Z},
\bZ)$ is the pullback of $H^1(C, \bZ)$.
This however implies that $$\wedge^4 H^1(\tilde{Z}, \bZ) = 0,$$ a contradiction
to the proven  injectivity of $\wedge^4 H^1(T, \bZ) \to
\wedge^4 H^1(\tilde{Z}, \bZ)$.

If $\tilde{Z}$ is of class
$VII$, then
$b_1(\tilde{Z}) = 1$ 
This again contradicts the surjectivity of
$$\pi_1(\tilde{Z}) \to \pi_1(T) \simeq \bZ^4$$ in Proposition \ref{prop4}.
\vskip .2cm 
Hence $f : X \to Y$ is smooth.\\

\hfill Q.E.D. for Theorem \ref{thm4}.

\vskip .3cm  Theorem \ref{main} now follows from Theorems \ref{thm0},
\ref{thm3} and
\ref{thm4}.

\section{Threefolds without meromorphic functions} 

Instead of assuming the existence of meromorphic functions we will require 
in this short concluding section the existence
of some holomorphic   vector fields or holomorphic 1-forms. 

\begin{theorem}\label{trivial1} Let $X$ be a smooth compact complex threefold 
which is homotopy equivalent to a torus. If the tangent bundle
$T_X$ is trivial, then $X$ is biholomorphic to a torus. 
\end{theorem}

\begin{proof} By our assumption (see  for instance \cite{Hu90}, page 144)
 $X \simeq G/\Gamma $, where
$G$ is a simply connected complex
$3$-dimensional Lie group  and $\Gamma \simeq \pi_1(X)$ cocompact. By Lemma \ref{group}, $G
\simeq \mathbb C^3$ or $G \simeq SL(2,\mathbb C)$ as groups. But $SL(2,\mathbb C)$ is not
contractible, e.g.
$$ b_3(SL(2,\mathbb C)) = 3$$
(see e.g. \cite{Ko59}),
hence $X \simeq \mathbb C^3/\mathbb Z^6,$ and our claim follows.

\end{proof}   

The paper \cite{Ko59} was brought to our attention by I.Radloff. 

For discussions concerning the first part of the following lemma, 
which is of course well-known to the experts, we would like to thank J.Winkelmann and in particular
A.Huckleberry.

\begin{lemma} \label{group}  
Let $G$ be a simply connected 3-dimensional complex Lie group. Then 
\begin{enumerate}
\item Either $G \simeq {\rm SL}(2,\bC)$ as Lie group or $G$ is solvable and 
$G \simeq \bC^3$ as complex manifold.
\item If $G$ is solvable and if $G$ contains a lattice $\Gamma$ 
(i.e., a discrete subgroup such that $G/\Gamma$ has bounded volume), such that
$\Gamma$ is abelian, then $G$ is abelian and therefore $G \simeq \bC^3$ as Lie group. 
\end{enumerate}   
\end{lemma}   

\begin{proof} (1) By the Levi-Malcev decomposition, $G$ is either semi-simple or solvable 
by reasons of dimension. 
In the semi-simple case, $G \simeq {\rm SL}(2,\bC).$  If $G$ is solvable, then
$G \simeq \bC^3 $ (even in any dimension), see e.g. [Na75], prop.1.4. \\
(2) Since $\Gamma $ is abelian, so does $G$ by e.g. [Wi98] (3.14.6). 
Hence $G \simeq \bC^3$ as Lie group. 

\end{proof}     
 
\begin{lemma} \label{a=0}  
Let $X$ be a compact  complex manifold with algebraic dimension
$ a(X) = 0$. Let $V$ be a holomorphic rank $r$ bundle on $X$: then the evaluation 
homomorphism 
$$  ev : H^0( X, V) \otimes  \hol_X  \ra V$$ 
is injective. In particular,   all the 
global sections of $V$ are carried by the  trivial rank $h$ subsheaf
$H^0( X, V) \otimes  \hol_X $.

In particular, $ h^0( X, V) = h \leq r$ and , if $h=r$ and $ det (V) \cong \hol_X$,
 then $V$ is trivial.
 
\end{lemma}   

\begin{proof}
For each point $x \in X$ we have a linear map of $\CC$-vector spaces 
$$  ev_x : H^0( X, V)  \ra V (x), $$ 
where $V(x) : = V_x / \mathfrak M_x V_x$ is the fibre of the vector bundle 
over the point $x$ ($\mathfrak M_x \subset \hol_x$ being the maximal ideal).

We claim that $ev_x$ is  injective for a general point $x \in X$. 

Otherwise, let $m$ be the generic rank of
$ev_x$: then we get a meromorphic map 
$$ k :  X \dashrightarrow  Grass (h-m, H^0( X, V) )$$
associating to $x$ the subspace $ker (ev_x)$.

By the projectivity of the Grassmann manifold, $k$ must be constant.
But a section vanishing at a general point is identically zero, which proves our assertion
that $ H^0( X, V) \otimes  \hol_X $ yields a subsheaf of $ V$.

Moreover, if $h=r$, the homomorphism $ev$ induces a non constant homomorphism
$ \Lambda^r (ev) : \hol_X \ra \det(V)$. Thus, if $\det (V)$ is trivial,
$ \Lambda^r (ev)$ is invertible, hence $ev$ is an isomorphism.

\end{proof}  

\begin{corollary}\label{n}
Let $X$ be a compact  complex manifold of dimension $n$ with algebraic dimension
$ a(X) = 0$ and with trivial canonical divisor  $K_X$. 
Then 
$$ h^0(\Omega^1_X) \geq n \Leftrightarrow h^0(\Theta_X) \geq n
 \Leftrightarrow \Theta_X \cong \hol_X^n \Leftrightarrow \Omega^1_X \cong \hol_X^n.$$       
\end{corollary}

\begin{theorem}\label{trivial2} Let $X$ be a smooth compact complex threefold with 
 trivial canonical divisor  $K_X$ which is homotopy equivalent to a torus. 

If  $h^0(X,\Omega^1_X) \geq 3$ or if $h^0(X,T_X) \geq 3,$ then
$X$ is biholomorphic to a torus.
\end{theorem} 

\begin{proof} 
By our Main Theorem 1.1 we may assume $a(X) = 0$, and applying the previous corollary \ref{n}
we get $\Omega^1_X \cong \hol_X^3.$ Now we conclude by Theorem \ref{trivial1}.
 
\end{proof} 

\begin{remark} {\rm It seems already difficult to exclude the case $h^0(\Omega_X^1) = 2.$ 
Taking a basis $\omega_1, \omega_2,$ 
we are able to exclude the case when both $\omega_i$ are non-closed. Since $X$ is not necessarily
K\"ahler, the existence of a   closed holomorphic 1-form does not lead to a non-trivial Albanese map,
which is the obstacle to conclude. }
\end{remark}

\vskip .2cm \noindent Fabrizio Catanese \\ Lehrstuhl Mathematik VIII,
Mathematisches Institut\\ Universit\"at
Bayreuth, D-95440 Bayreuth, Germany\\ fabrizio.catanese@uni-bayreuth.de

\vskip .2cm \noindent Keiji Oguiso \\ Department of Mathematics\\
Osaka University, Toyonaka 560-0043 Osaka, Japan\\
oguiso@math.sci.osaka-u.ac.jp

\vskip .2cm \noindent Thomas Peternell \\ Lehrstuhl Mathematik I,
Mathematisches Institut\\ Universit\"at Bayreuth,
D-95440 Bayreuth, Germany\\ thomas.peternell@uni-bayreuth.de


\vskip 1cm

\begin{thebibliography}{999999}

\bibitem[BHPV04]{BHPV04} Barth, W. P., Hulek, K., Peters, C. A. M.,
Van de Ven,A.: \textit{Compact complex surfaces},
Springer(2004).

\bibitem[Bl53]{Bl53} Blanchard, A.: \textit{Recherche de structures
analytiques complexes sur certaines
vari\'et\'es}, C. R. Acad. Sci., Paris,  S\'er. I, Math. {\bf 238}
(1953) 657--659.

\bibitem[Bl56]{Bl56} Blanchard, A.: \textit{Sur les variŽtŽs analytiques complexes.} 
 Ann. Sci.
Ecole Norm. Sup. (3)  73  (1956), 157--202.

\bibitem[Ca95]{Ca95} Catanese, F.: \textit{Compact complex manifolds
bimeromorphic to tori}, Abelian varieties
(Egloffstein, 1993)  de Gruyter, Berlin (1995) 55--62.

\bibitem[Ca02]{Ca02} Catanese, F.: \textit{Deformation types of real
and complex manifolds}, Contemporary trends in
algebraic geometry and algebraic topology (Tianjin, 2000) Nankai
Tracts Math. {\bf 5} (2002) 195--238.

\bibitem[Ca04]{Ca04} Catanese, F.: \textit{Deformation in the large
of some complex manifolds. I.}, Ann. Mat. Pura
Appl. {\bf 183} (2004) 261--289.

\bibitem[CKO03]{CKO03} Catanese, F., Keum, J.H., Oguiso. K.:
\textit{Some remarks on the universal cover of an open $K3$ surface},
Math. Ann. {\bf 325} (2003) 279--286.

\bibitem[Fj78]{Fj78} Fujiki, A.: \textit{On automorphism groups of compact
K\"ahler manifolds}, Invent. Math. {\bf 44}
(1978) 225--258.

\bibitem[Fu78]{Fu78} Fujita, T.: \textit{On K\"ahler fibre spaces
over curves}, J. Math. Soc. Japan {\bf 30} (1978)
779--794.

\bibitem[Fu86]{Fu86} Fujita, T.: \textit{Zariski decomposition and
canonical rings of elliptic threefolds}, J. Math.
Soc. Japan {\bf 38} (1986) 19--37.

\bibitem[Gr68]{Gr68} Griffiths, P. A.: \textit{Periods of integrals
on algebraic manifolds. II. Local study of the
period mapping}, Amer. J. Math.  {\bf 90} (1968) 805--865.

\bibitem[GHJ03]{GHJ03} Gross, M., Huybrechts, D., and Joyce, D.:
\textit{Calabi-Yau manifolds and related geometries},
Universitext. Springer-Verlag, Berlin (2003).

\bibitem[Hi75]{Hi75} Hironaka, H.: \textit{Flattening theorem in
complex-analytic geometry}, Amer. J. Math. {\bf 97}
(1975) 503--547.

\bibitem[Hi77]{Hi77} Hironaka, H.: \textit{Bimeromorphic smoothing of
a complex-analytic space}, Acta Math. Vietnam.
{\bf 2} (1977) 103--168.

\bibitem[Hu90]{Hu90}Huckelberry, A : \textit{Actions of groups of holomorphic transformations.}
Encycl. Math. Sciences vol. 69, ed. W.Barth, R.Narasimhan, 
Springer (1990),143--196.

\bibitem[Ka81]{Ka81} Kawamata, Y.: \textit{Characterization of
abelian  varieties}, Compositio Math. {\bf 43} (1981)
253--276.

\bibitem[Ko59]{Ko59} Kostant,B.: \textit{ The principal three-dimensional subgroup and the Betti numbers of
a complex simple Lie group}, Amer. J. Math. {\bf 81} (1959), 973--1032

\bibitem[Na75]{Na75} Nakamura,I.: \textit{Complex parallelizable manifold and their small deformations}, 
J. Diff. Geom. {\bf 10} (1975), 85--112 

\bibitem[No83]{No83} Nori, M. V.: \textit{Zariski's conjecture and
related problems}, Ann. Sci. \'Ecole Norm. Sup.
{\bf 16} (1983) 305--344.

\bibitem[So75]{So75} Sommese, A.J.: \textit{Quaternionic manifolds},
Math. Ann. {\bf 212} (1975) 191--214.

\bibitem[Ue75]{Ue75} Ueno,K.: \textit{Classification theory of
algebraic varieties and compact complex  spaces},
Lecture Notes in Mathematics {\bf 439} Springer-Verlag, Berlin-New York (1975).

\bibitem[Ue87]{Ue87} Ueno, K.: \textit{On compact analytic threefolds
with non trivial Albanese tori}, Math. Ann.
{\bf 278} (1987) 41--70.

\bibitem[Wi98]{Wi98} Winkelmann,J.: \textit{Complex analytic geometry of complex parallelizable 
manifolds}, M\'emoirs de la S.M.F.; {\bf 72-73} (1998)

\end{thebibliography}
\end{document}